\documentclass{amsart}
\usepackage[utf8]{inputenc}
\usepackage{todonotes}

\usepackage{geometry}
\usepackage{amssymb,latexsym,enumitem,bbm,xcolor,bm}
\usepackage[noadjust]{cite}
\usepackage{setspace}
\usepackage{color}
\usepackage{hyperref}

\newcommand{\R}{\mathbb{R}}
\newcommand{\N}{\mathbb{N}}
\newcommand{\Z}{\mathbb{Z}}

\newcommand{\vecu}{\mathbf{u}}
\newcommand{\vecv}{\mathbf{v}}

\newcommand{\Sc}{\mathcal{S}}
\newcommand{\Scsum}{\boldsymbol{\mathcal{S}}}

\newcommand{\Exp}{\mathbb{E}}

\newcommand{\inpr}[3][]{\left\langle#2 \,,\, #3\right\rangle_{#1}}

\newcommand{\oldinpr}[3][]{{\lll\!#2 \,,\, #3\!\ggg_{#1}}}

\newcommand{\oldnorm}[2][]{{|\!\|#2\|\!|_{#1}}}

\numberwithin{equation}{section}

\newtheorem{theorem}{Theorem}[section]
\newtheorem{corollary}[theorem]{Corollary}
\newtheorem{lemma}[theorem]{Lemma}
\newtheorem{proposition}[theorem]{Proposition}
\newtheorem{remark}[theorem]{Remark}

\newtheorem{definition}[theorem]{Definition}

\allowdisplaybreaks

\title{The Burgers' Equation in the Hermite-Sobolev Spaces}
\author{Suprio Bhar}
\address{Suprio Bhar, Department of Statistics and Data Science, Indian Institute of Technology Kanpur, Kalyanpur, Kanpur - 208016, India.}
\email{suprio@iitk.ac.in}

\author{Rajeev Bhaskaran}
\address{Rajeev Bhaskaran, Harish-Chandra Research Institute, Chhatnag Rd, Jhusi, Prayagraj, Uttar Pradesh - 211019, India.}
\email{brajeev58j@gmail.com}

\author{Barun Sarkar}
\address{Barun Sarkar, Department of Mathematics, Indian Institute of Technology Madras, Chennai - 600036, India.}
\email{barun@iitm.ac.in}

\date{}

\begin{document}
\begin{abstract}
   In this paper, we show existence and uniqueness of solutions to the viscous Burgers' equation in $\mathbb{R}^d$, when the initial condition $u_0$ is in Hermite-Sobolev space of index $p$, for suitable non-negative integers $p$. Our solutions are local in time. We also have a regularity result, viz. if $u_0$ belongs to Schwartz space, then so does the solution.
\end{abstract}


\maketitle

\section{Introduction}
The Burgers' equation (\cite{MR1146}) is an important partial differential equation that arises in fluid mechanics and related areas. In this paper, which is the third in a series of three papers (the other two being \cite{alt-mono, sp-algebra}), we prove the existence and uniqueness of solutions to the system of equations arising from Burgers' equation in dimensions $d \geq 1$ in the following form:
\[ \partial_t \vecu(t, x) = \frac{1}{2}\bigtriangleup \vecu(t, x) + \vecu(t, x). \nabla \vecu(t, x);~~\vecu(0,x) = \vecu_0(x),~~ t \geq 0, x \in \mathbb R^d.\] Here $\vecu(t,x) = (u_1(t,x),\cdots,u_d(t,x))$, where $u_i(t,x)$ are the components of the vector valued function $\vecu(t,x)$ and $\vecu_0(x) = (u_0^1(x),\cdots,u_0^d(x))$ is the given initial vector of functions. We have taken the convective term on the right hand side, which is easily seen to be equivalent to the standard form. We have taken the viscosity parameter $\nu = \frac{1}{2}$, but any $\nu > 0$ also works. We prove our results when the initial functions $u_0^i$ have some regularity that we describe below. 
                           
There is a vast literature on Burgers' equation based on the classical Sobolev spaces, which is available in the scholarly domain; see \cite{MR4198716} for some recent results on the Burgers' equation without the viscous term, see \cite{MR2318582} for a review of problems related to the Burgers' equation, see \cite[Subsection 4.4.1]{MR2597943} for the representation of solution using the well-known Cole-Hopf transformation (\cite{MR47234, MR42889}). However, our approach in this series of papers  uses a different family of spaces than the usual Sobolev spaces used in PDE theory viz. we prove our results in the setting of Hermite-Sobolev spaces $\Sc_p$, for non-negative integers $p$. These spaces were used by K. It\^o, G. Kallianpur and others as a framework to study stochastic partial differential equations (see \cite{MR1465436, MR771478}). In this paper we use these spaces to prove the existence and uniqueness of solutions for Burgers' equation. We use an essentially deterministic approach to prove that for given initial condition $\vecu_0$ with $u_{0}^i \in \mathcal S_p, p > 1+ \frac{d+1}{2}$, there exists unique solutions $\vecu(t)$ with $u_i(t) \in \mathcal S_p$ upto an explosion time $\tau > 0$. We recall here that for $p > \frac{d}{4}+1$ the $\mathcal S_p$ spaces consist of continuously differentiable functions \cite[Theorem 4.1]{MR2373102}.

An essential feature of our approach, different from the techniques in \cite{MR1465436, MR771478} is the use of norms different from, but equivalent to the original Hermite-Sobolev norms for $p$ non-negative integers (see Section 2 below). These norms are particularly well suited to handle the `product type' non-linearities that occur in equations such as the Burgers' equation. In \cite{alt-mono} we established the key `Monotonicity inequalities' \cite{MR570795} in these alternate norms, for differential operators in the context of Stochastic PDEs, which are easily adapted to the PDE context in the current paper. These are essential to show existence and uniqueness for certain classes of linear Stochastic PDEs connected with finite dimensional diffusions as shown in \cite{MR2590157}. Also essential for handling the non-linearities arising in Burgers' equation is the boundedness of the multiplication by functions, shown in \cite{sp-algebra}.

Our approach is to first prove existence and uniqueness for $\vecu_0$ with $u_0^i \in \Sc$, the Schwarz space of rapidly decreasing smooth functions on $\R^d.$ We do this in Section 3. Since the `Monotonicity inequality' implies the continuity of solutions with respect to the initial conditions and since $\mathcal S$ is dense in $\mathcal S_p$ we get the required extension of results in Section 3 to the case when $u_0 \in \mathcal S_p, p > 1 + \frac{d+1}{2}$. To prove our existence results in Section 3 we use solutions of the linearised equation (see \eqref{nwclclnr1}) to generate a sequence of linear approximations by iteration (see \eqref{iteration-linear}). To show convergence of this sequence we use an analog of the `two-step Monotonicity inequality' established earlier (\cite{brajeev-arxiv}) for (scalar-type) non-linear operators in the context of SPDEs. To obtain bounds on the approximating sequence, we use a stopping technique used there, suitably modified to handle the norms in the current situation. Finally uniqueness is proved again using a `Monotonicity inequality' for the non-linear operator in Burgers' equation.

\section{Preliminaries: Hilbertian Topology on Schwartz space}

Let $\mathbb{Z}^d_+:=\{\alpha=(\alpha_1,\cdots, \alpha_d): \; \alpha_i \text{ non-negative integers}\}$. If $\alpha\in\mathbb{Z}^d_+$, we define $|\alpha|:=\alpha_1+\cdots+\alpha_d$. The topology on $\Sc(\R^d)$ is given by a family of seminorms $p_n, n = 0, 1, \cdots$ (see \cite{MR771478, MR1681462}) where
\[p_n(f) := \sup_{x \in \R^d} \left[(1 + |x|)^n \max_{\alpha : |\alpha| \leq n} \left|\frac{\partial^{|\alpha|}}{\partial_{x_1}^{\alpha_1}\cdots \partial_{x_d}^{\alpha_d}}f(x) \right|\right], \forall f \in \Sc(\R^d)\]
where $|x|$ denotes the usual Euclidean norm for $x \in \R^d$. We shall adopt the short-hand notation $\partial^\alpha$ and $\partial_j$ for $\frac{\partial^{|\alpha|}}{\partial_{x_1}^{\alpha_1}\cdots \partial_{x_d}^{\alpha_d}}$ and $\frac{\partial}{\partial_{x_j}}$, respectively.

For $p \in \R$, consider the increasing family of norms $\oldnorm[p]{\cdot}$, defined by the inner products
\begin{equation}
\oldinpr[p]{f}{g}
:=\sum_{n\in\mathbb{Z}^d_+}(2|n|+d)^{2p}\langle f,h_n\rangle_0 \langle g,h_n\rangle_0,\ \ \ f,g\in\Sc(\R^d).
\end{equation}
In the above equation, $\{h_n: n\in\mathbb{Z}^d_+\}$ is an orthonormal basis for $\mathcal{L}^2(\R^d)$ given by the Hermite functions and $\langle\cdot,\cdot\rangle_0$ is the inner product in $\mathcal{L}^2(\R^d)$ with respect to the Lebesgue measure. The Hermite-Sobolev spaces $\Sc_p(\R^d), p \in \R$ are defined as the completion of $\Sc(\R^d)$ in
$\oldnorm[p]{\cdot}$. Note that the dual space $\Sc_p^\prime(\R^d)$ is isometrically isomorphic with $\Sc_{-p}(\R^d)$ for $p\geq 0$. For $\phi \in \Sc(\R^d)$ and $\psi \in \Sc^\prime(\R^d)$, we write the duality action by $\inpr{\psi}{\phi}$. The same notation $\inpr{\psi}{\phi}$ shall also stand for the duality between $\phi \in \Sc_p(\R^d)$ and $\psi \in \Sc_{-p}(\R^d)$. We also have $\Sc(\R^d) = \bigcap_{p}(\Sc_p(\R^d),\oldnorm[p]{\cdot}), \Sc^\prime(\R^d) = \bigcup_{p>0}(\Sc_{-p}(\R^d),\oldnorm[-p]{\cdot})$ and $\Sc_0(\R^d) = \mathcal{L}^2(\R^d)$. The following basic relations hold for the $\Sc_p(\R^d)$ spaces: for $0<q<p$, \[\Sc(\R^d) \subset \Sc_p(\R^d) \subset \Sc_q(\R^d) \subset \mathcal L^2(\R^d) = \Sc_0(\R^d) \subset \Sc_{-q}(\R^d)\subset \Sc_{-p}(\R^d)\subset \Sc^\prime(\R^d).\]
The topology on $\Sc(\R^d)$ given by the norms $\oldnorm[p]{\cdot}, p \in \Z_+$ is the same as the topology on $\Sc(\R^d)$ generated by the seminorms $p_n, n = 0, 1, \cdots$ (see \cite[Proposition 1.1]{MR1837298}).

We now recall an alternate way to describe the same topology on $\Sc(\R^d)$. For $p \in \Z_+$, define for $\phi, \psi \in \Sc(\R^d)$,
\[\inpr[p]{\phi}{\psi}:= \sum_{|\alpha| + |\beta| \leq 2p}\int_{\R} x^\alpha\partial^\beta \phi(x)\, x^\alpha\partial^\beta \psi(x)\, dx,\]
where $\alpha,\beta\in\Z^d_+$ in the above sum. We denote the corresponding norms by $\|\cdot\|_p$. It is known that the topology on $\Sc(\R^d)$ generated by $\|\cdot\|_p, p \in \Z_+$ is the same as the topology generated by the seminorms $p_n, n = 0, 1, \cdots$. We recall the next result from \cite[Proposition 3.3]{MR1999259} (also see \cite[Remark 1.3.1]{MR771478}).

\begin{proposition}\label{norm-equivalence}
For all $p \in \Z_+$, there exist constants $C_1 = C_1(d, p) > 0$ and $C_2 = C_2(d, p) > 0$ such that
\[\oldnorm[p]{\phi} \leq C_1 \|\phi\|_p \leq C_2 \oldnorm[p]{\phi}, \forall \phi \in \Sc(\R^d).\]
\end{proposition}

Consequently, for $p \in \Z_+$, completing $\Sc(\R^d)$ with the inner-product $\inpr[p]{\cdot}{\cdot}$ gives us the same Hermite-Sobolev spaces $\Sc_p(\R^d)$ as when we complete it with respect to the other inner product $\oldinpr[p]{\cdot}{\cdot}$. Unless stated otherwise, for $p \in \Z_+$ we shall work with the norm $\|\cdot\|_p$. We shall write $\Sc$, $\Sc^\prime$ and $\Sc_p$, instead of $\Sc(\R^d)$, $\Sc^\prime(\R^d)$ and $\Sc_p(\R^d)$, respectively, for convenience of notation. The dimension $d$ shall be clear from the context.

\begin{theorem}[Product in $\Sc_p$ spaces, {\cite[Theorem 2.8]{sp-algebra}}]\label{product-Sp}
Let $p$ be any integer with $p > \frac{d+1}{2}$. Then for all $\phi_1, \phi_2 \in \Sc_p$,
\begin{equation}\label{product-continuity}
\|\phi_1 \phi_2\|_p \leq C \|\phi_1\|_p \|\phi_2\|_p,    
\end{equation}
for some positive constant $C = C_{p, d}$. In particular, $\Sc_p$ is an algebra.
\end{theorem}

\begin{remark}
For $u \in \Sc_p$, we have $\partial_j u \in \Sc_{p-1}$. To make sense of the product $u. \partial u$ in $\Sc_{p - 1}$, we need $p - 1 > \frac{d+1}{2}$, or equivalently, $p > \frac{d+1}{2} + 1$. We shall use this lower bound on $p$ several times in the subsequent discussion.
\end{remark}

Let $\sigma_{ij}, b_i, \, 1 \leq i, j \leq d$ be bounded $C^\infty(\R^d)$ functions with bounded derivatives. Consider the operators on $\Sc^\prime$,
\begin{equation}\label{L-A}
\begin{split}
A \phi &:= (A_1 \phi, A_2\phi, \dots, A_d\phi),\\
A_i\phi &:= \sum_{j=1}^d\sigma_{ji} \partial_j \phi, i = 1, 2, \cdots, d,\\
L\phi &:= \frac{1}{2}\sum_{i,j=1}^d ( \sigma \sigma^t )_{ij} \partial^2_{ij} \phi + \sum_{j=1}^d b_j \partial_j \phi.
\end{split}
\end{equation}
In the next result, we also use the Hilbert-Schmidt norm $\|A\phi\|_{HS(p)}^2 := \sum_{i = 1}^d \|A_i\phi\|_{p}^2$ for $\phi \in \Sc^\prime$, whenever the right hand side is defined.

\begin{theorem}[Monotonicity Inequality {\cite[Theorem 2.4]{alt-mono}}]\label{L2-alt-mono}
The pair of operators $(L, A)$ satisfy the Monotonicity inequality in $\|\cdot\|_p$ in the sense that
\begin{equation}\label{Monotoniticity-inequality}
2\inpr[p]{\phi}{L\phi} + \|A\phi\|_{HS(p)}^2\leq C\|\phi\|^2_p,\, \forall \phi\in \Sc
\end{equation}
where $C = C(d, p, \sigma, b) > 0$ is a positive constant not depending on $\phi$. Moreover, the dependence of $C = C(d, p, \sigma, b) > 0$ on $\sigma, b$ is only through the supremum of the derivatives up to order $2p$.

\end{theorem}

\section{The Burgers' equation}

We write $\Scsum := \oplus_{k = 1}^d \Sc(\R^d)$ and $\Scsum_p := \oplus_{k = 1}^d \Sc_p(\R^d)$. Here, and in what follows, we use the boldface notation $\vecu$ for elements of $\Scsum$ and $\Scsum_p$.

Consider the equation
\begin{equation}\label{Burger-PDE}
\frac{\partial\vecu (t,\cdot)}{\partial t} =  \frac{1}{2}\bigtriangleup\, \vecu(t,\cdot) + \vecu(t,\cdot) . \nabla\, \vecu(t,\cdot),\, t \geq 0
\end{equation}
with $\vecu(0, \cdot) = \vecu_0(\cdot) = (u_{01}(\cdot), \cdots, u_{0d}(\cdot)) \in \Scsum \subset \Scsum_p$ and $\vecu(t,\cdot) . \nabla\, \vecu(t,\cdot) = \sum_{j = 1}^d u_j(t,\cdot) \partial_j\, u(t,\cdot)$. The inner-product and norm on $\Scsum_p$ are defined in the usual way, i.e. for $\vecu = (u_1, \cdots, u_d), \vecv = (v_1, \cdots, v_d) \in \Scsum_p$, 
\[\inpr[p]{\vecu}{\vecv} := \sum_{j = 1}^d \inpr[p]{u_j}{v_j},\]
and
\[\|\vecu\|^2_p := \sum_{j = 1}^d \|u_j\|^2_p.\]
$\Scsum_p$ becomes a real separable Hilbert space with associated inner-product $\inpr[p]{\cdot}{\cdot}$ and norm $\|\cdot\|_p$.

\begin{definition}[$\Scsum_p$ and $\Scsum$ valued local Solution to \eqref{Burger-PDE}]
We say $(\vecu, \tau)$ is an $\Scsum_p$ valued local solution to \eqref{Burger-PDE} with $p - 1 \geq \frac{d+1}{2}$ and with initial $\vecu_0 \in \Scsum_p$, if 
\begin{enumerate}
    \item $\tau > 0$,
    
    \item $\vecu : [0, \tau) \to \Scsum_p$ is continuous, 

    \item the function $t \mapsto \left( \frac{1}{2}\bigtriangleup\, \vecu(t,\cdot) + \vecu(t,\cdot) . \nabla\, \vecu(t,\cdot) \right)$ on $[0, \tau)$ is Bochner integrable in $\Scsum_{p - 1}$,

    \item the following integral equation is satisfied in $\Scsum_{p-1}$:
\begin{equation}\label{Burger}
\vecu(t,\cdot) = \vecu_0 + \int_0^t \left( \frac{1}{2}\bigtriangleup\, \vecu(s,\cdot) + \vecu(s,\cdot) . \nabla\, \vecu(s,\cdot) \right)\, ds , \, t < \tau.
\end{equation}
\end{enumerate}

Moreover, we say 
$(\vecu, \tau)$ is an $\Scsum$ valued local solution to \eqref{Burger-PDE} if it is an $\Scsum_p$ valued solution for all $p$ with $p - 1 \geq \frac{d+1}{2}$.
\end{definition}

\begin{definition}[Maximal solution]
    An $\Scsum_p$ valued local solution $(\vecu, \tau)$ to \eqref{Burger-PDE} is said to be maximal if $\lim_{t \uparrow \tau}\|\vecu(t)\|_p = \infty$.
\end{definition}

In dimension 1, explicit solutions to the viscous Burgers' equation has been discussed in \cite[p. 207, subsection 4.4.1.b]{MR2597943}.

\subsection{Linearized Equation:}
We first linearize the given equation and study the existence and uniqueness of solutions to the linearized equation. 

Given $\bar \vecv(\cdot) \in \Scsum$ and a continuous $\mathbf{b}: [0, \infty) \to  \Scsum$, consider the linear equation in unknown $\vecv$ \begin{equation}\label{nwclclnr1}
 \vecv(t,\cdot)= \bar \vecv(\cdot) + \int_0^t \left( \frac{1}{2} \bigtriangleup \vecv(s,\cdot) +  \mathbf{b}(s,\cdot) \cdot \nabla \vecv(s,\cdot) \right)\, ds 
 \end{equation}

We discuss below results on the existence and uniqueness of solutions to \eqref{nwclclnr1}. This can be directly obtained from \cite[Theorem 1]{MR2590157} and Theorem \ref{L2-alt-mono}, keeping $A \equiv 0$. However, we describe below an approach using a linear Stochastic PDE in order to obtain a stochastic representation of the solution to \eqref{nwclclnr1}, yielding more information about the structure of the solution.

\begin{theorem}[Existence and Uniqueness of Stochastic PDE version of \eqref{nwclclnr1}, {\cite[Theorem 1]{MR2590157}}]\label{GMR-existence-uniqueness}
The Stochastic PDE
\begin{equation}\label{linear-SPDE}
\mathbf{V}(t,\cdot)= \bar \vecv(\cdot) + \int_0^t \left( \frac{1}{2} \bigtriangleup \mathbf{V}(s,\cdot) +  \mathbf{b}(s,\cdot)\cdot  \nabla \mathbf{V}(s,\cdot) \right)\, ds + \int_0^t \nabla \mathbf{V}(s, \cdot). dB_s, \, t \geq 0.   
\end{equation}
has a unique strong solution, where $\{B_t\}_{t \geq 0}$ is a $d$-dimensional standard Brownian motion. 
\end{theorem}
\begin{proof}
For any fixed $T > 0$, the function $\mathbf{b}: [0, T] \to  \Scsum$ is continuous by our assumption. In particular, for any $p > 0$, $\mathbf{b}: [0, T] \to  \Scsum_p$ is also continuous and hence
\[\sup_{t \in [0, T]} \|b(t, \cdot)\|_p < \infty.\]
Using \cite[Theorem 4.1]{MR2373102}, 
\[\sup_{t \in [0, T]} \|\partial^\alpha b(t, \cdot)\|_{\mathcal{L}^\infty(\R^d)} < \infty.\]
We then obtain the existence and uniqueness of solutions to the linear Stochastic PDE \eqref{linear-SPDE} on $[0, T]$ using \cite[Theorem 1]{MR2590157} and Theorem \ref{L2-alt-mono}. The result for the time interval $[0, \infty)$ follows by patching up the solutions on finite time intervals $[0, T]$.
\end{proof}

The existence and uniqueness of solutions to \eqref{nwclclnr1} is obtained from the solution in Theorem \ref{GMR-existence-uniqueness} as follows.

\begin{corollary}[Existence and Uniqueness of \eqref{nwclclnr1}]\label{corollary-existence-uniquness-PDE}
Under the above assumptions, the PDE \eqref{nwclclnr1} has an $\Scsum$ valued unique solution.
\end{corollary}

\begin{proof}
Taking expectation on both sides of equation \eqref{linear-SPDE}, interpreted component-wise, we have
\[\Exp \mathbf{V}(t,\cdot)= \bar \vecv(\cdot) + \int_0^t \left( \frac{1}{2} \bigtriangleup\, \Exp \mathbf{V}(s,\cdot) +  \mathbf{b}(s,\cdot)\cdot  \nabla\, \Exp \mathbf{V}(s,\cdot) \right)\, ds, \, t \geq 0,\]
which shows $\Exp \mathbf{V}(\cdot)$ is a solution to \eqref{nwclclnr1}. Here, the integrability of $\Exp \mathbf{V}(\cdot)$ and other terms follow from \cite[Lemma 1]{MR2590157}. Note that $\Exp \mathbf{V}(\cdot)$ is in $\Scsum_p$ for all $p$ and hence is in $\Scsum$. This shows the existence of an $\Scsum$ valued solution.

To establish the uniqueness, let $\vecv$ and $\tilde \vecv$ be two $\Scsum$ valued solutions. Then, for any positive integer $p$ with $p - 1 > \frac{d+1}{2}$,
\begin{align*}
\|\vecv(t) - \tilde \vecv(t)\|_{p-1}^2 = \int_0^t \inpr[p-1]{\vecv(s) - \tilde \vecv(s)}{\bigtriangleup(\vecv(s) - \tilde \vecv(s))\,  + 2 \mathbf{b}(s,\cdot)\cdot  \nabla(\vecv(s) - \tilde \vecv(s))}\, ds.
\end{align*}
Here, we use the fact that $\mathbf{b} \in \Scsum$ and hence, the product term $\mathbf{b}(s,\cdot)\cdot  \nabla(\vecv(s) - \tilde \vecv(s))$ makes sense in $\Scsum$.
The uniqueness follows from Theorem \ref{L2-alt-mono}, together with an application of Gronwall's inequality.
\end{proof}

\subsection{Burgers' Equation: The Non-Linear Case}
Given $\vecu_0 \in \Scsum$, we iteratively define continuo  us $\vecu_n:[0, \infty)] \to \Scsum, n = 1, 2, \cdots$ by the following iteration scheme. If $\vecu_n$ is known, then define $
\vecu_{n+1}$ via the unique solution to the linear PDEs
\begin{equation}\label{iteration-linear}
\frac{\partial \vecu_{n+1}}{\partial t}(t) = \frac{1}{2} \bigtriangleup \vecu_{n+1}(t) + \vecu_n(t) \cdot\nabla \vecu_{n+1}(t), t > 0; \quad \vecu_{n+1}(0) = \vecu_0.
\end{equation}
In the above equation, as $\vecu_n(t)$ is $\Scsum$ valued, the product term $\vecu_n(t) \cdot\nabla \vecu_{n+1}(t)$ makes sense by Theorem \ref{product-Sp}. Here, by a solution we mean that $\vecu_{n+1}$ satisfies
\begin{equation}\label{iteration-integral-form}
\vecu_{n+1}(t) = \vecu_0 +  \int_0^t \left[\frac{1}{2}\bigtriangleup \vecu_{n+1}(s) + \vecu_n(s) \cdot\nabla \vecu_{n+1}(s)\right]\, ds, \, t \geq 0.
\end{equation}
The existence and uniqueness of solutions of \eqref{iteration-integral-form} follows from Corollary \ref{corollary-existence-uniquness-PDE} and the solution $\vecu_{n+1}:[0, \infty) \to \Scsum$ is continuous.

To prove existence of a solution to \eqref{Burger-PDE}, we need to obtain bounds on the difference of successive solutions $\vecu_{n+1} - \vecu_n$. We do this by applying the two-step monotonicity inequality, proved in Lemma 3.6 below. Since the constant in the two-step Monotonicity inequality depends on the bounds of the derivatives, we need to control the distances of various derivatives of $\vecu_n$'s from the corresponding derivatives of $\vecu_0$ by restricting the time intervals on which these are defined, i.e. for $|\alpha| \leq 2p$, we need to control the norm of $\partial^\alpha \vecu_n(t) - \partial^\alpha \vecu_0$ in $\Scsum_q$ for a suitable $q\leq p$. We do this via certain stopping times, which we introduce below (see also \cite{brajeev-arxiv}).

Fix a positive integer $p$ and consider the function $K : \{0, 1, \cdots, 2p\} \to \N$ defined by $ K(0) := 1$ and $K(k) := \lfloor\frac{k+1}{2}\rfloor$ otherwise. Here, $\lfloor r \rfloor$ denotes the greatest integer less or equal to a real number $r$.

Consider for any fixed $\lambda > 0$ and for any positive integer $n$, and for multi-indices $\alpha$ with $0 \leq |\alpha| \leq 2p$, define
\begin{align}\label{stoptime}
\begin{split}
\tau_{n, \alpha}^{p, \lambda} &:= \inf\left\{t > 0 \mid \|\partial^\alpha \vecu_n(t) - \partial^\alpha \vecu_0\|_{p-K(|\alpha|)} \geq \lambda \right\} \wedge T,\\
\tau_n^{p, \lambda} &:= \left( \bigwedge_{\alpha: |\alpha| = 0}^{2p} \tau_{n, \alpha}^{p, \lambda} \right) \wedge T \wedge \tau_{n-1}^{p, \lambda}.
\end{split}
\end{align}
The above stopping times $\tau_n^{p, \lambda}$ yield norm-estimates for $\partial^\alpha \vecu_n(t) - \partial^\alpha \vecu_0$, which are used in Theorem \ref{existence-Burger} in establishing the convergence of $\vecu_n$'s by estimating the differences $\vecu_{n+1} - \vecu_n$.

By definition, $\tau_n^{p, \lambda}$ is
\begin{enumerate}[label=(\roman*)]
    \item non-decreasing in $\lambda$, for fixed $n$ and $p$
    \item non-increasing in $n$, for fixed $\lambda$ and $p$
    \item non-increasing in $p$, for fixed $n$ and $\lambda$
\end{enumerate}
We shall look at $\lim_{n\to\infty}\tau_n^{p, \lambda}=: \tau^{p, \lambda}$ and $\lim_{p \to \infty} \tau^{p, \lambda} =:\tau^\lambda$. By construction,
\begin{enumerate}[label=(\roman*)]
    \item $\tau^{p, \lambda}$ is non-decreasing in $\lambda$, for fixed $p$
    \item $\tau^{p, \lambda}$ is non-increasing in $p$, for fixed $\lambda$
    \item $\tau^{\lambda}$ is non-decreasing in $\lambda$.
\end{enumerate}

\begin{proposition}\label{local-existence}
    We have $\tau^{p, \lambda}, \tau^{\lambda} \in (0, T]$ and consequently $\tau := \lim_{\lambda \to \infty} \tau^\lambda \in (0, T]$.
\end{proposition}

\begin{proof}
We first show that $\tau^{p, \lambda} > 0$.

If possible, let $\tau^{p, \lambda} = \lim_{n\to\infty}\tau_n^{p, \lambda} = 0$. Then, there exists a subsequence $\{n_k\}_k$ such that $\tau_{n_k}^{p, \lambda} < \tau_{n_k-1}^{p, \lambda}$, $k\geq2$ and a multi-index $\alpha$ with $0 \leq |\alpha| \leq 2p$ such that
\[\|\partial^\alpha 
\vecu_{n_k}(\tau_{n_k}^{p, \lambda}) - \partial^\alpha \vecu_0\|_{p-K(|\alpha|)} = \lambda.\]
Here, we claim that,
\begin{align*}
 \|\partial^\alpha \vecu_{n_k}(\tau_{n_k}^{p, \lambda}) - \partial^\alpha \vecu_0\|_{p-K(|\alpha|)} \leq C_\lambda\, \tau_{n_k}^{p, \lambda}. 
\end{align*}
If we assume the claim, then $\left\{\frac{\|\partial^\alpha \vecu_{n_k}(\tau_{n_k}^{p, \lambda}) - \partial^\alpha \vecu_0\|_{p-K(|\alpha|)}}{\tau_{n_k}^{p, \lambda}} \right\}_k$ is a bounded sequence, with the numerators appearing in the terms being upper bounded. However, this contradicts $\tau_{n_k}^{p, \lambda} \downarrow \tau^{p, \lambda} = 0$. Hence, we must have $\tau^{p, \lambda} > 0$.

We first prove the above claim for $\alpha \equiv 0$. Start by assuming
\[\|\vecu_{n_k}(\tau_{n_k}^{p, \lambda}) - \vecu_0\|_{p-1} = \lambda.\]
Due to the bounds imposed by stopping, using \eqref{iteration-linear} and the norm estimate in Theorem \ref{product-Sp}, we have,
\begin{equation}\label{main-estimate}
\|\vecu_{n_k}(\tau_{n_k}^{p, \lambda}) - \vecu_0\|_{p-1} \leq \int_0^{\tau_{n_k}^{p, \lambda}} \left[ \frac{1}{2} \|\bigtriangleup \vecu_{n_k}(s)\|_{p-1} + \|\vecu_{n_k+1}(s) .\nabla \vecu_{n_k}(s)\|_{p-1} \right] \, ds \leq C_\lambda\, \tau_{n_k}^{p, \lambda},
\end{equation}
for some $C_\lambda > 0$. This proves the claim for $\alpha \equiv 0$. For a general $\alpha$ with $0 \leq |\alpha| \leq 2p$, consider the equation satisfied by the derivative term $\partial^\alpha \vecu_{n+1}$. Applying $\partial^\alpha$ on both sides of \eqref{iteration-linear}, we have
\[\frac{\partial \partial^\alpha \vecu_{n+1}}{\partial t}(t) = \frac{1}{2} \bigtriangleup\partial^\alpha \vecu_{n+1}(t) + \partial^\alpha \left(\vecu_n(t). \nabla \vecu_{n+1}(t)\right), t > 0; \quad \partial^\alpha \vecu_{n+1}(0) = \partial^\alpha \vecu_0.\]
Due to the bounds imposed by stopping, we can obtain the relevant bound of the term $\|\partial^\alpha \vecu_{n_k}(\tau_{n_k}^{p, \lambda}) - \partial^\alpha \vecu_0\|_{p-K(|\alpha|)}$ by looking at the corresponding integral equation as in \eqref{main-estimate}.

Now, we look at the decreasing limit $\lim_{p \to \infty} \tau^{p, \lambda} = \tau^\lambda$. Note that, by construction, $\tau^{p, \lambda}$ is decreasing in $p$ for every fixed $\lambda > 0$. An argument similar to the above, but for a sequence $\{\tau^{p_k, \lambda}\}_k$ instead of $\{\tau_{n_k}^{p, \lambda}\}_k$, shows that the terms $\frac{\|\vecu_k(\tau^{p_k, \lambda}) - \vecu_0\|_{p_k-1}}{\tau^{p_k, \lambda}}$ still remain bounded, leading to a contradiction if we assume $\tau^\lambda = 0$. Hence, $\tau^\lambda > 0$ for any $\lambda > 0$.

By construction, $\tau^\lambda$ is increasing in $\lambda$. Hence, $0 < \tau \leq T$.
\end{proof}
Now, we show that our proposed iteration scheme converges in appropriate norms $\|\cdot\|_p$ and yields a solution to \eqref{Burger} locally in the corresponding $\Scsum_p$ space. The following two-step monotonicity inequality is used in the proof of existence, similar to the use of a two-step monotonicity inequality to obtain the existence of certain Stochastic PDEs (see \cite[Theorem 3.3 and Theorem 4.3]{brajeev-arxiv}).

\begin{lemma}\label{Mon-inq2}
Fix $R > 0$. For any $p>\frac{d+1}{2}$ and $\bm\phi,\bm\psi,\bm\xi\in B_p(R):=\{\bm\varphi\in\Scsum_p:\|\bm\varphi\|_p\leq R\}$, 
\begin{align}\label{Mon-inq2-eqns}
\begin{split}
& \inpr[p-1]{\bm\phi-\bm\psi}{\bigtriangleup(\bm\phi-\bm\psi)} + 2\inpr[p-1]{\bm\phi-\bm\psi}{\bm\psi. \nabla\bm\phi- \bm\xi. \nabla\bm\psi}\\
& \quad \leq C(R,p,d)\, \|\bm\phi-\bm\psi\|_{p-1}^2 + \|\bm\psi-\bm\xi\|_{p-1}^2,
\end{split}
\end{align}
where $C(R,p,d)$ is a positive constant.  
\end{lemma}

\begin{proof} Since $\inpr[p]{\vecu}{\vecv} = \sum_{j = 1}^d \inpr[p]{u_j}{v_j}$, we may apply the inequality \eqref{Mon-inq2-eqns} for $d = 1$ to each of terms in the sum to get the inequality in the general case. Hence, it suffices to establish the inequality for $d = 1$. Now,
\begin{align}\label{moneqn2-cal1}
\begin{split}
& \inpr[p-1]{\phi-\psi}{\partial^2(\phi-\psi)} + 2 \inpr[p-1]{\phi-\psi}{\psi\, \partial\phi- \xi\, \partial\psi}\\
& = \inpr[p-1]{\phi-\psi}{\partial^2(\phi-\psi)} + 2 \inpr[p-1]{\phi-\psi}{\psi\, \partial (\phi-\psi)} + 2 \inpr[p-1]{\phi-\psi}{(\psi-\xi)\, \partial \psi}.
\end{split}
\end{align}
The first term on the RHS of \eqref{moneqn2-cal1}, is handled by a constant-coefficient monotonicity inequality, (see \cite[Theorem 3.1]{MR3331916}), i.e.
\begin{align}\label{moneqn2-cal2}
\begin{split}
\inpr[p-1]{\phi-\psi}{\partial^2(\phi-\psi)} & \leq 2\inpr[p-1]{\phi-\psi}{\frac{1}{2}\partial^2(\phi-\psi)} + \|\partial(\phi-\psi)\|_{p-1}^2 \\
& \leq C(p)\, \| \phi-\psi\|_{p-1}^2.
\end{split}
\end{align}
For the second term on the RHS of \eqref{moneqn2-cal1}, we use computations involving terms of the form $\inpr[p-1]{\phi}{f\partial \phi}$ from \cite[proof of Theorem 2.1]{alt-mono} and obtain
\begin{equation}\label{moneqn2-cal3}
\left|  2 \inpr[p-1]{\phi-\psi}{\psi\, \partial (\phi-\psi)}\right| \leq C(p,R) \| \phi-\psi\|_{p-1}^2.
\end{equation}
For the third term on the RHS of \eqref{moneqn2-cal1},
\begin{align}\label{moneqn2-cal4}
\begin{split}
& \left| 2 \inpr[p-1]{\phi-\psi}{(\psi-\xi)\, \partial \psi}\right| \\
& \leq 2\, \|\phi-\psi\|_{p-1} \|(\psi-\xi)\, \partial \psi\|_{p-1}\\
& \leq 2\, C(p)\, \|\phi-\psi\|_{p-1} \|\psi-\xi\|_{p-1} \|\partial \psi\|_{p-1},\ \text{[by Theorem \ref{product-Sp}]}\\
& \leq C^2(p)\, \|\partial \psi\|_{p-1}^2 \|\phi-\psi\|^2_{p-1} + \|\psi-\xi\|_{p-1}^2 \\
& \leq C(p,R)\, \|\phi-\psi\|^2_{p-1} + \|\psi-\xi\|_{p-1}^2,\, \text{[$\because$ $\|\partial\psi\|_{p-1}\leq C_p\|\psi\|_p\leq C(p,R)$, as $\psi\in B_p(R)$]}.
\end{split}
\end{align}
Therefore, putting \eqref{moneqn2-cal4}, \eqref{moneqn2-cal3},  \eqref{moneqn2-cal2} into  \eqref{moneqn2-cal1}, we obtain our result.

\end{proof}

\begin{theorem}[Existence of a solution]\label{existence-Burger}
For any $p > \frac{d+1}2 + 1$, the sequence $\{\vecu_n\}_n$ converges uniformly on $[0, \tau^{p, \lambda}]$ in $\Scsum_{p-1}$ to a limit $\vecu^{p, \lambda}$ and $(\vecu^{p, \lambda}, \tau^{p, \lambda})$ is an $\Scsum_{p-1}$ valued local solution to \eqref{Burger}.
\end{theorem}

\begin{proof}
In this proof, we use a two-step monotonicity argument, similar to \cite[Theorem 3.3]{brajeev-arxiv}.

Fix $t \in [0, \tau^{p, \lambda}]$. Then, since $\tau^{p, \lambda} \leq \tau_n^{p, \lambda}, \forall n$, we have $\|\vecu_n(t)\|_{p - 1} \leq \|\vecu_0\|_{p - 1} + \|\vecu_n(t) - \vecu_0\|_{p - 1} \leq \|\vecu_0\|_{p - 1} + \lambda$ and hence,
\begin{align*}
&\|\vecu_{n+1}(t) - \vecu_n(t)\|_{p-1}^2\\
&= \int_0^t \inpr[p-1]{\vecu_{n+1}(s) - \vecu_n(s)}{\bigtriangleup(\vecu_{n+1}(s) - \vecu_n(s))} \, ds\\
& + 2\int_0^t \inpr[p-1]{\vecu_{n+1}(s) - \vecu_n(s)}{\vecu_{n}(s). \nabla \vecu_{n+1}(s) - \vecu_{n-1}(s). \nabla \vecu_{n}(s)} \, ds \\
&\text{(using the bounds on the derivatives of $\vecu_n(s)$ during $[0, \tau^{p, \lambda}_n]$}\\
&\text{and the resulting two-step Monotonicity inequality from Lemma \ref{Mon-inq2})},\\
& \leq C_{\lambda, p, \vecu_0} \int_0^t \|\vecu_{n+1}(s) - \vecu_n(s)\|_{p-1}^2\, ds + \int_0^t \|\vecu_n(s) - \vecu_{n-1}(s)\|_{p-1}^2\, ds\\
& \leq C(\lambda,p,\vecu_0,t) \int_0^t \|\vecu_n(s) - \vecu_{n-1}(s)\|_{p-1}^2\, ds\ \ \text{[by Gronwall's lemma]} \\
& = C(\lambda,p,\vecu_0,t) \int_0^t \|\vecu_n(t_1) - \vecu_{n-1}(t_1)\|_{p-1}^2\, dt_1\\
& \leq C^2(\lambda,p,\vecu_0,t) \int_0^t \int_0^{t_1}\|\vecu_{n-1}(t_2) - \vecu_{n-2}(t_2)\|_{p-1}^2\, dt_2dt_1\\
& \leq C^n(\lambda,p,\vecu_0,t) \int_0^t \int_0^{t_1} \int_0^{t_2}\cdots\int_0^{t_{n-1}}\|\vecu_1(t_n)-\vecu_0(t_n)\|_{p-1}^2\, dt_n\cdots dt_3dt_2dt_1\\
& \leq \lambda^2C^n(\lambda,p,\vecu_0,t) \frac{t^n}{n!}\ \ \text{[from \eqref{stoptime}]}.
\end{align*}
Now observe that,
 \[\vecu_{n+1}(t) = \vecu_0 + \sum_{k=0}^n\left(\vecu_{k+1}(t)-\vecu_k(t) \right).\]
Then,
\begin{align*}
\left\| \vecu_m(t)-\vecu_n(t)\right\|_{p-1} & = \left\| \sum_{k=n}^{m-1}(\vecu_{k+1}(t)-\vecu_k(t))\right\|_{p-1} \\
& \leq \sum_{k=n}^{m-1} \left\| \vecu_{k+1}(t)-\vecu_k(t)\right\|_{p-1} \\
& =  \sum_{k=n}^{m-1} \left\| \vecu_{k+1}(t)-\vecu_k(t)\right\|_{p-1}\, k\cdot \frac{1}{k}\\
& \leq \left( \sum_{k=n}^{m-1} \left\| \vecu_{k+1}(t)-\vecu_k(t)\right\|^2_{p-1}\, k^2\right)^{1/2} \left(  \sum_{k=n}^{m-1}\frac{1}{k^2}\right)^{1/2} \\
& \leq \left( \sum_{k=n}^{m-1} \lambda^2C^k(\lambda,p,\vecu_0,t) \frac{t^k}{k!} \, k^2\right)^{1/2} \left(  \sum_{k=1}^{\infty}\frac{1}{k^2}\right)^{1/2}.
\end{align*}
Consequently
\[\sup_{t\in[0, \tau^{p, \lambda}]} \left\| \vecu_m(t)-\vecu_n(t)\right\|_{p-1}\leq \left( \sum_{k=n}^{m-1} \lambda^2C^k(\lambda,p,\vecu_0,\tau^{p,\lambda}) \frac{(\tau^{p,\lambda})^k}{k!} \, k^2\right)^{1/2} \left(  \sum_{k=1}^{\infty}\frac{1}{k^2}\right)^{1/2}.\]
Now $\sum_{k=0}^{\infty} \lambda^2C^k(\lambda,p,\vecu_0,\tau^{p,\lambda}) \frac{(\tau^{p,\lambda})^k}{k!} \, k^2$, converges and hence the tail sum 
\[\sum_{k=n}^{m-1} \lambda^2C^k(\lambda,p,\vecu_0,\tau^{p,\lambda}) \frac{(\tau^{p,\lambda})^k}{k!} \, k^2\to 0,\] 
as $m,n\to\infty$. And $\sum_{k=1}^{\infty}\frac{1}{k^2}$ converges. Therefore $\{\vecu_m(t):t\in[0, \tau^{p, \lambda}]\}_m$ is Cauchy in $\Scsum_{p-1}$ and hence converges uniformly on $[0, \tau^{p, \lambda}]$ in $\Scsum_{p-1}$. We denote this limit by $\vecu^{p, \lambda}$. 

Taking term-by-term limit in \eqref{iteration-integral-form}, we observe that $\vecu^{p, \lambda}$ solves \eqref{Burger} on $[0, \tau^{p, \lambda}]$.
\end{proof}

We now discuss a monotonicity inequality required to obtain the uniqueness of a solution to \eqref{Burger}.

\begin{lemma}\label{nonlinear-mono}
Fix $R > 0$. For any $p>\frac{d+1}{2}$ and $\bm\phi,\bm\psi\in B_p(R)$, 
\begin{equation}
\inpr[p-1]{\bm\phi-\bm\psi}{\bigtriangleup(\bm\phi-\bm\psi)} + 2\inpr[p-1]{\bm\phi-\bm\psi}{\bm\phi. \nabla\bm\phi- \bm\psi. \nabla\bm\psi} \leq C(R,p,d)\, \|\bm\phi-\bm\psi\|_{p-1}^2,
\end{equation}
where $C(R,p,d)$ is a positive constant.
\end{lemma}

\begin{proof}
The proof is similar to Lemma \ref{Mon-inq2}. The term $\inpr[p-1]{\bm\phi-\bm\psi}{\bigtriangleup(\bm\phi-\bm\psi)}$ is estimated using \eqref{moneqn2-cal2}. For $2 \inpr[p-1]{\bm\phi-\bm\psi}{\bm\phi. \nabla\bm\phi- \bm\psi. \nabla\psi}$, note that
\begin{align*}
&2 \inpr[p-1]{\bm\phi-\bm\psi}{\bm\phi. \nabla\phi- \bm\psi\, \nabla\bm\psi}\\
= &2 \inpr[p-1]{\bm\phi-\bm\psi}{\bm\phi. \nabla (\bm\phi-\bm\psi)} + 2 \inpr[p-1]{\bm\phi-\bm\psi}{(\bm\phi-\bm\psi). \nabla \bm\psi}\\
= &2 \sum_{j=1}^d \inpr[p-1]{\phi_j-\psi_j}{(\bm\phi. \nabla (\bm\phi-\bm\psi))_j} + 2 \sum_{j=1}^d \inpr[p-1]{\phi_j-\psi_j}{((\bm\phi-\bm\psi). \nabla \bm\psi)_j},
\end{align*}
where $(\bm\phi. \nabla (\bm\phi-\bm\psi))_j$ denotes the $j$-th component of $\bm\phi. \nabla (\bm\phi-\bm\psi)$. 
The terms in the first sum on the right hand side are estimated similar to \eqref{moneqn2-cal3}, and the terms in the second sum are estimated by the Cauchy-Schwarz inequality, Proposition \ref{product-Sp} and the boundedness of $\partial_j : \Sc_p \to \Sc_{p - 1}, j = 1, 2, \cdots, d$. This completes the proof.
\end{proof}

Using Lemma \ref{nonlinear-mono}, we have a consistency statement involving the solutions of \eqref{Burger}.
\begin{lemma}\label{consistency}
For integers $p, q > \frac{d+1}2 + 1$ and for any $\lambda, \eta > 0$, we have
$\vecu^{p, \lambda} = \vecu^{q, \eta}$ on $[0, \tau^{p, \lambda} \wedge \tau^{q, \eta}]$.
\end{lemma}

\begin{proof}
Without loss of generality, we take $p < q$. Then, 
\[\vecu^{p, \lambda}(t) = \vecu_0 +  \int_0^t \left[\frac{1}{2}\bigtriangleup \vecu^{p, \lambda}(s) + \vecu^{p, \lambda}(s). \nabla \vecu^{p, \lambda}(s)\right]\, ds,\]
and
\[\vecu^{q, \eta}(t) = \vecu_0 +  \int_0^t \left[\frac{1}{2}\bigtriangleup \vecu^{q, \eta}(s) + \vecu^{q, \eta}(s). \nabla \vecu^{q, \eta}(s)\right]\, ds,\]
hold for $t \in [0, \tau^{p, \lambda} \wedge \tau^{q, \eta}]$, with equality in $\Scsum_{p-1}$. Then,
\begin{align*}
    \| \vecu^{p, \lambda}(t) - \vecu^{q, \eta}(t)\|_{p-1}^2 &= \int_0^t \inpr[p-1]{ \vecu^{p, \lambda}(s) -  \vecu^{q, \eta}(s)}{\bigtriangleup(\vecu^{p, \lambda}(s) - \vecu^{q, \eta}(s))}\, ds\\
    &+2\int_0^t \inpr[p-1]{\vecu^{p, \lambda}(s) - \vecu^{q, \eta}(s)}{ \vecu^{p, \lambda}(s).\nabla \vecu^{p, \lambda}(s) - \vecu^{q, \eta}(s).\nabla \vecu^{q, \eta}(s)}\, ds.
\end{align*}
The required equality follows from the Monotonicity inequality (Lemma \ref{nonlinear-mono}) and Gronwall's inequality.
\end{proof}

By patching up $\Scsum_p$ valued local solutions, we show the existence of an $\Scsum$ valued local solution.
\begin{theorem}\label{defn-by-consistency}
Assume $\vecu_0 \in \Scsum$. Then, there exists an $\Scsum$-valued local solution $(\vecu, \tau)$ for \eqref{Burger}.
\end{theorem}

\begin{proof}
Fix $\lambda > 0$. Recall that $\tau^{p, \lambda} \downarrow \tau^\lambda$ as $p \to \infty$. By Lemma \ref{consistency}, $\vecu^{p, \lambda} = \vecu^{q, \lambda}$ on $[0, \tau^{\lambda}]$ for any $p, q > \frac{d+1}2 + 1$. As $\vecu^{p, \lambda}$ on $[0, \tau^{\lambda}]$ does not depend on $p$, we denote this function $\vecu^\lambda$, by dropping the super-script $p$. But, by construction, $\vecu^\lambda$ is $\Scsum_{p-1}$ valued for all $p> \frac{d+1}2 + 1$ and hence is $\Scsum$ valued and $(\vecu^\lambda, \tau^\lambda)$ is an $\Scsum$ valued local solution to \eqref{Burger}.

Again, by Lemma \ref{consistency}, we have $\vecu^{p, \lambda} = \vecu^{p, \eta}$ on $[0, \tau^ \lambda) \cap [0, \tau^ \eta)$ for any $\lambda, \eta > 0$ and for all integers $p > \frac{d+1}2 + 1$. Consequently, we have $\vecu^\lambda = \vecu^\eta$ on $[0, \tau^ \lambda) \cap [0, \tau^ \eta)$ for any $\lambda, \eta > 0$. Since $\tau^\lambda \uparrow \tau$ as $\lambda \to \infty$, define $\vecu:[0, \tau) \to \Scsum$ by $\vecu(t) := \vecu^\lambda(t)$ if $t \in [0, \tau^ \lambda)$. By construction,  $(\vecu, \tau)$ is an $\Scsum$ valued local solution to \eqref{Burger}.
\end{proof}

We now establish the uniqueness of the $\Scsum$ valued local solution.
\begin{theorem}\label{existence-uniqueness-schwartz-initial}
Let $\vecu_0 \in \Scsum$.
The local solution to Burgers' equation \eqref{Burger} obtained earlier (Theorem \ref{defn-by-consistency}) is unique, in the following sense: if $(\hat\vecu, \tau_1)$ and $(\hat\vecv, \tau_2)$ are two local solutions with the same initial value $\vecu_0$, then $\hat\vecu(t) = \hat\vecv(t), \forall t < \tau_1 \wedge \tau_2$.
\end{theorem}
\begin{proof}
Following the estimates obtained regarding the existence argument in Theorem \ref{existence-Burger}, for any two solutions $\hat \vecu$ and $\hat\vecv$, on their common interval of existence, using Lemma \ref{nonlinear-mono}, we have for all $t \leq \eta^{p, \lambda} \wedge \tau_1 \wedge \tau_2$,
\begin{align*}
    \|\hat\vecu(t) - \hat\vecv(t)\|_{p-1}^2 &= \int_0^t \inpr[p-1]{\hat\vecu(s) - \hat\vecv(s)}{\bigtriangleup(\hat\vecu(s) - \hat\vecv(s))}\, ds\\
    &+2\int_0^t \inpr[p-1]{\hat\vecu(s) - \hat\vecv(s)}{\hat\vecu(s). \nabla\hat\vecu(s) - \hat\vecv(s). \nabla\hat\vecv(s)}\, ds\\
    &\leq C_{\lambda, p,\vecu_0} \int_0^t \|\hat\vecu(s) - \hat\vecv(s)\|_{p-1}^2\, ds,
\end{align*}
where
\begin{align*}
\eta_{\alpha}^{p, \lambda} &:= \inf\left\{t > 0 \mid \|\partial^\alpha \hat\vecu(t) - \partial^\alpha\vecu_0\|_{p-K(|\alpha|)} \geq \lambda\, \text{or}\, \|\partial^\alpha \hat\vecv(t) - \partial^\alpha \vecu_0\|_{p-K(|\alpha|)} \geq \lambda\right\} \wedge T,\\
\eta^{p, \lambda} &:= \left( \bigwedge_{\alpha:|\alpha| = 0}^{2p} \eta_{\alpha}^{p, \lambda} \right) \wedge T,
\end{align*}
for any large positive $\lambda$ and any integer $p > \frac{d+1}{2} + 1$. Using Gronwall's inequality, we have $\hat\vecu(t) = \hat\vecv(t), \forall t \leq \eta^{p,\lambda} \wedge \tau_1 \wedge \tau_2$ in $\Scsum_{p-1}$ and hence in $\Scsum$. Letting $\lambda$ go to infinity, we get the required uniqueness.
\end{proof}

\begin{remark}
In Proposition \ref{local-existence}, we have $\tau \in (0, T]$. If $\tau = T$, then the system \eqref{Burger-PDE} can be started with the initial condition $\vecu(T)$ and extended to a larger time interval. Now, if $\tau < T$ at any stage, it is a reasonable question to ask if $\lim_{t \uparrow \tau}\|\vecu(t)\|_p = \infty$ for some $p > \frac{d+1}{2} + 1$. Using the definition of $\tau^{p, \lambda}_n$, for any $\lambda > 0$ and large $p$, there exists $N = N(\lambda, p)$ such that for each $n \geq N$, there exists a multi-index $\alpha$ with $0 \leq |\alpha| \leq 2p$ with 
\[\|\partial^{\alpha_n} \vecu_n(\tau^{\bar p, \lambda}_n) - \partial^{\alpha_n} \vecu_0\|_{p - K(|\alpha|)} = \lambda.\]
We would be done if we could take $\lambda \to \infty$, keeping $p$ fixed. However, this would require the equality of $\tau^\prime := \lim_{p \to \infty} \lim_{\lambda \to \infty} \tau^{p, \lambda}$ with $\tau = \lim_{\lambda \to \infty} \lim_{p \to \infty} \tau^{p, \lambda}$, which we have been unable to establish.
\end{remark}

\section{Continuity with respect to Initial value}

Consider the two following equations.
\begin{align*} 
& \vecu_{1}(t) = \vecu_{0,1} +  \int_0^t \left[\frac{1}{2}\bigtriangleup \vecu_{1}(s) + \vecu_1(s) . \nabla \vecu_{1}(s)\right]\, ds,\\
& \vecu_{2}(t) = \vecu_{0,2} +  \int_0^t \left[\frac{1}{2}\bigtriangleup \vecu_{2}(s) + \vecu_2(s) .\nabla \vecu_{2}(s)\right]\, ds,
\end{align*}
where the initial conditions $\vecu_{0,1}, \vecu_{0,2} \in \Scsum$. Then we have the following estimate.

\begin{lemma}\label{continuity-estimate}
For any positive $\lambda$ and any positive integer 
$p > \frac{d+1}2 + 1$, with
\begin{align*}
\eta_{\alpha}^{p, \lambda} &:= \inf\left\{t > 0 \mid \|\partial^\alpha \vecu_1(t) - \partial^\alpha \vecu_{0,1}\|_{p-K(|\alpha|)} \geq \lambda \text{ or } \|\partial^\alpha \vecu_2(t) - \partial^\alpha \vecu_{0,2}\|_{p-K(|\alpha|)} \geq \lambda\right\} \wedge T,\\
\eta^{p, \lambda} &:= \left( \bigwedge_{\alpha : |\alpha| = 0}^{2p} \eta_{\alpha}^{p, \lambda} \right) \wedge T ,
\end{align*}
we have for $t \leq \eta^{p, \lambda}$
\[\left\| \vecu_{1}(t) - \vecu_{2}(t)\right\|_{p-1}^2 \leq \left\|  \vecu_{0,1} -  \vecu_{0,2}\right\|_{p-1}^2 e^{Ct},\]
where $C = C_{\lambda, p, \vecu_{0,1}, \vecu_{0,2}}$ is a positive constant. Moreover, $C$ depends on $\vecu_{0,1}, \vecu_{0,2}$ only through $\|\vecu_{0,1}\|_p, \|\vecu_{0,2}\|_p$ and is non-decreasing in these parameters.
\end{lemma}
\begin{proof}
By It\^o's formula, for all $t \leq \eta^{p, \lambda}$, 
\begin{align*}
\left\| \vecu_{1}(t) - \vecu_{2}(t)\right\|_{p-1}^2 & = \left\|  \vecu_{0,1} -  \vecu_{0,2}\right\|_{p-1}^2 \\
& \quad + \int_0^t \inpr[p-1]{\vecu_{1}(s) - \vecu_{2}(s)}{\bigtriangleup\left( \vecu_{1}(s) - \vecu_{2}(s)\right)}\, ds \\
& \quad + \int_0^t 2\inpr[p-1]{\vecu_{1}(s) - \vecu_{2}(s)}{\vecu_1(s). \nabla \vecu_{1}(s) - \vecu_2(s) .\nabla \vecu_{2}(s)}\, ds \\
&\leq \left\|  \vecu_{0,1} -  \vecu_{0,2}\right\|_{p-1}^2+ C_{\lambda, p, \vecu_{0,1}, \vecu_{0,2}} \int_0^t \| \vecu_1(s) - \vecu_2(s)\|_{p-1}^2\, ds\\
& \leq \left\|  \vecu_{0,1} -  \vecu_{0,2}\right\|_{p-1}^2 e^{C_{\lambda, p, \vecu_{0,1}, \vecu_{0,2}}t}.
\end{align*}
In the above argument we have used Lemma \ref{nonlinear-mono} and Gronwall's inequality. This completes the proof.
\end{proof}
\begin{theorem}
Let the initial value $\vecu_0\in \Scsum_p$ for some $p > \frac{d+1}{2} + 1$. Then there exists a unique $\Scsum_{p-1}$ valued local solution $(\vecu, \tau)$ to \eqref{Burger}.
\end{theorem}
\begin{proof}
Let $\{\vecu_{0,n}\}_{n\in\N}$ be a sequence such that $\vecu_{0,n}\in \Scsum$, for each $n\in\N$ and $\vecu_{0,n}\to \vecu_0$ in $\Scsum_p$ i.e. $\|\vecu_{0,n}-\vecu_0\|_p\to0$. Where $\vecu_{0,n}$ is the initial condition of the following non-linear PDE 
\begin{equation}\label{u-n-initialconv1}
\vecu_{n}(t) = \vecu_{0,n} +  \int_0^t \left[\frac{1}{2}\bigtriangleup \vecu_{n}(s) + \vecu_n(s) .\nabla \vecu_{n}(s)\right]\, ds, \, t \geq 0,
\end{equation}
and $\vecu_0\in \Scsum_p$ is the initial condition of \eqref{Burger}. Define
\begin{align}\label{stoptime-u-n-initial}
\begin{split}
\sigma_{n, \alpha}^{p, \lambda} &:= \inf\left\{t > 0 \mid \|\partial^\alpha \vecu_n(t) - \partial^\alpha \vecu_{0,n}\|_{p-K(|\alpha|)} \geq \lambda \right\} \wedge T,\\
\sigma_n^{p, \lambda} &:= \left( \bigwedge_{\alpha: |\alpha| = 0}^{2p} \sigma_{n, \alpha}^{p, \lambda} \right) \wedge T \wedge \sigma_{n-1}^{p, \lambda}.
\end{split}
\end{align}
For integers $m \leq n$ and for $t\leq \sigma_n^{p, \lambda}$, arguing similar to Theorem \ref{existence-Burger}, using Lemma \ref{continuity-estimate} we have
\[\left\| \vecu_n(t) - \vecu_{m}(t)\right\|_{p-1}^2 \leq \left\| \vecu_{0,n} - \vecu_{0,m}\right\|_{p-1}^2 e^{C(\lambda,p,\vecu_{0, n}, \vecu_{0, m})t} \leq \left\| \vecu_{0,n} - \vecu_{0,m}\right\|_{p-1}^2 e^{C(\lambda,p,\vecu_{0})t}.\]
In the last inequality we have used the facts that $\|\vecu_{0,k} - \vecu_0\|_p\to0$ as $k \to \infty$ and $e^{C(\lambda,p,\vecu_{0, n}, \vecu_{0, m})}$ depends on $\vecu_{0, n}, \vecu_{0, m}$ only through their $p$-norm. 

Now observe that,
 \[\vecu_{n+1}(t)=\vecu_0 + \sum_{k=0}^n\left(\vecu_{k+1}(t)-\vecu_k(t) \right).\]
Since $\vecu_{0,n}\to \vecu_0$ in $\Scsum_p$ as $n \to \infty$, arguing as in the proof of Theorem \ref{existence-Burger}, we have the convergence of $\{\vecu_n(t)\}_n$ to a limit in $\Scsum_{p-1}$, say $\vecu^{p, \lambda}(t)$, uniformly in $t\leq \sigma^{p, \lambda}:= \lim_{n\to\infty}\sigma_n^{p, \lambda}$.

By construction $\sigma_n^{p, \lambda}\downarrow\sigma^{p, \lambda}$. Now we claim that $\sigma^{p, \lambda}>0$. The arguments are similar to Proposition \ref{local-existence} and we only provide a sketch of the proof.

If possible, let $\sigma^{p, \lambda} = 0$. Then, there exists a subsequence $\{n_k\}_k$ such that $\sigma_{n_k}^{p, \lambda} < \sigma_{n_k-1}^{p, \lambda}$, $k\geq2$ and a multi-index $\alpha$ with $0 \leq |\alpha| \leq 2p$ such that
\[\|\partial^\alpha \vecu_{n_k}(\sigma_{n_k}^{p, \lambda}) - \partial^\alpha \vecu_{0, n_k}\|_{p-K(|\alpha|)} = \lambda.\]
As done in Proposition \ref{local-existence}, we can arrive at a bounded sequence $\left\{\frac{\|\partial^\alpha \vecu_{n_k}(\sigma_{n_k}^{p, \lambda}) - \partial^\alpha \vecu_{0, n_k}\|_{p-K(|\alpha|)}}{\sigma_{n_k}^{p, \lambda}} \right\}_k$, leading to a contradiction.

As done in Lemma \ref{consistency}, we can show that any two solutions $\vecu^{p, \lambda}(t)$ and $\vecu^{q, \lambda}(t)$ existing upto $\sigma^{p, \lambda}$ and $\sigma^{q, \lambda}$ are equal on the common time-interval $[0, \sigma^{p, \lambda} \wedge \sigma^{q, \lambda}]$ of existence. Due to this consistency, we can repeat the construction undertaken in Theorem \ref{defn-by-consistency} and finally construct the unique (as in Theorem \ref{existence-uniqueness-schwartz-initial}) local solution to \eqref{Burger} in $\Scsum_{p-1}$ with initial condition $\vecu_0\in\Scsum_p$. 
\end{proof}

\bibliography{ref}

\begin{thebibliography}{10}

\bibitem{MR2318582}
J\'er\'emie Bec and Konstantin Khanin.
\newblock Burgers turbulence.
\newblock {\em Phys. Rep.}, 447(1-2):1--66, 2007.

\bibitem{sp-algebra}
Suprio Bhar and Rajeev Bhaskaran.
\newblock Products and {C}onvolutions in the {H}ermite-{S}obolev spaces.
\newblock {\em arXiv}, 2026 (http://arxiv.org/abs/2607.29061).

\bibitem{alt-mono}
Suprio Bhar, Rajeev Bhaskaran, and Arvind~Kumar Nath.
\newblock Existence and uniqueness of stochastic {PDE}s associated with the
  forward equations: an approach using alternate norms.
\newblock {\em Infin. Dimens. Anal. Quantum Probab. Relat. Top.}, 29(1):Paper
  No. 2550007, 21, 2026.

\bibitem{MR3331916}
Suprio Bhar and B.~Rajeev.
\newblock Differential operators on {H}ermite {S}obolev spaces.
\newblock {\em Proc. Indian Acad. Sci. Math. Sci.}, 125(1):113--125, 2015.

\bibitem{MR1146}
J.~M. Burgers.
\newblock Mathematical examples illustrating relations occurring in the theory
  of turbulent fluid motion.
\newblock {\em Verh. Nederl. Akad. Wetensch. Afd. Natuurk. Sect. 1}, 17(2):53,
  1939.

\bibitem{MR42889}
Julian~D. Cole.
\newblock On a quasi-linear parabolic equation occurring in aerodynamics.
\newblock {\em Quart. Appl. Math.}, 9:225--236, 1951.

\bibitem{MR2597943}
Lawrence~C. Evans.
\newblock {\em Partial differential equations}, volume~19 of {\em Graduate
  Studies in Mathematics}.
\newblock American Mathematical Society, Providence, RI, second edition, 2010.

\bibitem{MR1681462}
Gerald~B. Folland.
\newblock {\em Real analysis}.
\newblock Pure and Applied Mathematics (New York). John Wiley \& Sons, Inc.,
  New York, second edition, 1999.
\newblock Modern techniques and their applications, A Wiley-Interscience
  Publication.

\bibitem{MR2590157}
L.~Gawarecki, V.~Mandrekar, and B.~Rajeev.
\newblock The monotonicity inequality for linear stochastic partial
  differential equations.
\newblock {\em Infin. Dimens. Anal. Quantum Probab. Relat. Top.},
  12(4):575--591, 2009.

\bibitem{MR47234}
Eberhard Hopf.
\newblock The partial differential equation {$u_t+uu_x=\mu u_{xx}$}.
\newblock {\em Comm. Pure Appl. Math.}, 3:201--230, 1950.

\bibitem{MR771478}
Kiyosi It\^{o}.
\newblock {\em Foundations of stochastic differential equations in
  infinite-dimensional spaces}, volume~47 of {\em CBMS-NSF Regional Conference
  Series in Applied Mathematics}.
\newblock Society for Industrial and Applied Mathematics (SIAM), Philadelphia,
  PA, 1984.

\bibitem{MR1465436}
Gopinath Kallianpur and Jie Xiong.
\newblock {\em Stochastic differential equations in infinite-dimensional
  spaces}, volume~26 of {\em Institute of Mathematical Statistics Lecture
  Notes---Monograph Series}.
\newblock Institute of Mathematical Statistics, Hayward, CA, 1995.
\newblock Expanded version of the lectures delivered as part of the 1993
  Barrett Lectures at the University of Tennessee, Knoxville, TN, March 25--27,
  1993, With a foreword by Balram S. Rajput and Jan Rosinski.

\bibitem{MR570795}
N.~V. Krylov and B.~L. Rozovski\u{\i}.
\newblock Stochastic evolution equations.
\newblock In {\em Current problems in mathematics, {V}ol. 14 ({R}ussian)},
  pages 71--147, 256. Akad. Nauk SSSR, Vsesoyuz. Inst. Nauchn. i Tekhn.
  Informatsii, Moscow, 1979.

\bibitem{MR1837298}
B.~Rajeev.
\newblock From {T}anaka's formula to {I}to's formula: distributions, tensor
  products and local times.
\newblock In {\em S\'{e}minaire de {P}robabilit\'{e}s, {XXXV}}, volume 1755 of
  {\em Lecture Notes in Math.}, pages 371--389. Springer, Berlin, 2001.

\bibitem{brajeev-arxiv}
B.~Rajeev.
\newblock Translation invariant diffusions and stochastic partial differential
  equations in $\mathcal{S}^{\prime}$.
\newblock {\em arXiv:1901.00277v2 [math.PR]}, pages 1--56, 2019.

\bibitem{MR1999259}
B.~Rajeev and S.~Thangavelu.
\newblock Probabilistic representations of solutions to the heat equation.
\newblock {\em Proc. Indian Acad. Sci. Math. Sci.}, 113(3):321--332, 2003.

\bibitem{MR2373102}
B.~Rajeev and S.~Thangavelu.
\newblock Probabilistic representations of solutions of the forward equations.
\newblock {\em Potential Anal.}, 28(2):139--162, 2008.

\bibitem{MR4198716}
Denis Serre.
\newblock Source-solutions for the multi-dimensional {B}urgers equation.
\newblock {\em Arch. Ration. Mech. Anal.}, 239(1):95--116, 2021.

\end{thebibliography}
\bibliographystyle{plain}

\end{document}